\documentclass[12pt]{amsart}
\usepackage{mathrsfs}
\usepackage{a4wide}
\usepackage{amssymb}
\usepackage{amsthm}
\usepackage{amsmath}
\numberwithin{equation}{section}
\theoremstyle{plain}

\newtheorem{theorem}{Theorem}
\newtheorem{lemma}[theorem]{Lemma}
\newtheorem{corollary}[theorem]{Corollary}

\newcommand{\R}{\mathbb{R}}
\newcommand{\Q}{\mathbb{Q}}
\newcommand{\Z}{\mathbb{Z}}
\newcommand{\eps}{\epsilon}
\newcommand{\C}{\mathbb{C}}

\newcommand{\Fp}{\mathbb{F}_p}
\newcommand{\Fl}{\mathbb{F}_{\ell}}

\newcommand{\del}{\delta}
\newcommand{\Del}{\Delta}

\newcommand{\spmod}{\hspace{-8pt}\pmod}

\newcommand{\Et}{\widetilde{E}}
\newcommand{\Mt}{\widetilde{M}}
\newcommand{\hcal}{\mathcal{H}}
\newcommand{\St}{\widetilde{S}}

\newcommand{\al}{\alpha}

\DeclareMathOperator{\gal}{Gal}
\DeclareMathOperator{\tr}{Tr }
\DeclareMathOperator{\frob}{Frob}
\DeclareMathOperator{\SL}{SL}
\DeclareMathOperator{\GL}{GL}

\theoremstyle{remark}

\newtheorem*{remark}{Remark}
\newtheorem*{example}{Example}
\numberwithin{theorem}{section} \numberwithin{equation}{section}

\begin{document}
\title[Congruence Properties of Borcherds Product Exponents] {Congruence Properties of Borcherds Product Exponents}

\author{Keenan Monks} 

\author{Sarah Peluse} 

\author{Lynnelle Ye} 

\address{Keenan Monks, 73 N James St, Hazleton, PA 18201}
\email{monks@harvard.edu}
\address{Sarah Peluse, 491 Parkview Terrace, Buffalo Grove, IL 60089}
\email{peluse@uchicago.edu}
\address{Lynnelle Ye, P.O. Box 16820, Stanford, CA 94309}
\email{lynnelle@stanford.edu}
\thanks{The authors are grateful for the NSF's support of the REU at Emory University.} 

\begin{abstract}
In his striking 1995 paper, Borcherds~\cite{Borcherds} found an infinite product expansion for certain modular forms with CM divisors. In particular, this applies to the Hilbert class polynomial of discriminant $-d$ evaluated at the modular $j$-function. Among a number of powerful generalizations of Borcherds' work, Zagier made an analogous statement for twisted versions of this polynomial. He proves that the exponents of these product expansions, $A(n,d)$, are the coefficients of certain special half-integral weight modular forms. We study the congruence properties of $A(n,d)$ modulo a prime $\ell$ by relating it to a modular representation of the logarithmic derivative of the Hilbert class polynomial.
\end{abstract}
\maketitle

\section{Introduction and statement of results}

The modular $j$-invariant, one of the most important functions in modern number theory, has $q$-expansion $$j(z)=q^{-1}+744+196884q+21493760q^2+\cdots.\hspace{.5in} (q=e^{2\pi iz})$$ Singular moduli are the values of $j(z)$ at imaginary quadratic arguments. To be more precise, let $Q_d$ be the set of positive definite binary quadratic forms $Q(x,y)$ of discriminant $-d<0$, and let $\alpha_Q$ be the unique root in the upper half-plane of some $Q(x,1)\in Q_d$. Given one of these roots $\alpha_Q$, the Hilbert class polynomial of discriminant $-d$, 
$\hcal_d(z)$, is defined to be the minimal polynomial of $j(\al_Q)$ (with minor modifications if $d/3$ or $d/4$ is a square). The roots of $\hcal_d(z)$ generate the Hilbert class fields of imaginary quadratic fields. 

The study of the form $\hcal_d(j(z))$ has a long history. Until recently, it was notoriously difficult to compute Hilbert class polynomials.
Borcherds~\cite{Borcherds} was the first to describe an infinite product expansion for $\hcal_d(j(z))$,
\begin{equation}
\hcal_d(j(z))=\prod_{Q\in Q_d/\SL_2(\Z)}(j(z)-j(\alpha_Q))^{\frac{1}{\omega_Q}}=q^{-h(d)}\prod_{n=1}^{\infty}(1-q^n)^{A_d(n)}
\end{equation}
where $h(d)$ is the Hurwitz-Kronecker class number and $\omega_Q$ is the weight of $Q$ (see \cite{Zagier}). Zagier~\cite{Zagier} generalized these to a setting where Borcherds products and their twisted analogues are easier to compute. He defines, for all discriminants $D>0$ and $-d<0$, a twisted Hilbert class polynomial by  
\begin{equation}\label{twists}
\hcal_{D,d}(j(z))=\prod_{Q\in Q_{Dd}/\SL_2(\Z)}(j(z)-j(\alpha_Q))^{\chi(Q)}=\prod_{n=1}^{\infty}\left(\prod_{k=1}^{D-1}(1-\zeta_D^kq^n)^{\left(\frac Dk\right)}\right)^{A_{D,d}(n)}
\end{equation}
where $\zeta_D=e^{2\pi i/D}$ and $\chi(Q)=\left(\frac{-d}{p}\right)$ for some prime $p$ represented by $Q$ that does not divide $dD$.

In light of this, we study the natural question of the distribution of the exponents $A_d(n)$ among residue classes modulo $\ell$ for various primes $\ell$, and we show that these exponents often possess unexpected properties. To this end, let 
\begin{equation}
\del_{d}(t,\ell;X)=\frac{\#\{p<X : p\text{ is prime and }A_d(p)\equiv t\spmod{\ell}\}}{\pi(X)}
\end{equation}
where $\pi(X)$ is the number of primes less than $X$. For certain choices of $d$ depending on $\ell$, we will be able to compute the asymptotic value of $\del_{d}(t,\ell;X)$ for each $t\in\Z/\ell\Z$. For example, for $\ell=11$ and $d=4$, we have the following:
\begin{table}[htbp]\label{asymptotic}\caption{}
\begin{flushleft}
\begin{tabular}{|r|c|c|c|c|c|c|}
\hline
$X$& $\del_{4}(0,11;X)$& $\del_{4}(1,11;X)$& $\del_{4}(2,11;X)$& $\del_{4}(3,11;X)$&$ \del_{4}(4,11;X)$&$ \del_{4}(5,11;X)$\\
\hline
$10^4$       &$.0829$&$.0928$&$.0887$&$.0911$&$.0862$&$.0903$\\
$10^5$     &$.0908$&$.0927$&$.0853$&$.0898$&$.0883$&$.0888$\\
$10^6$  &$.0899$&$.0897$&$.0891$&$.0915$&$.0887$&$.0894$\\
$10^7$&$.0898$&$.0909$&$.0901$&$.0903$&$.0897$&$.0895$\\
$2\cdot10^7$&$.0899$&$.0905$&$.0901$&$.0903$&$.0902$&$.0896$\\
$5\cdot10^7$&$.0899$&$.0902$&$.0901$&$.0900$&$.0902$&$.0897$\\
\vdots& \vdots&\vdots &\vdots &\vdots &\vdots &\vdots\\
$\infty$  & $\frac{9}{100}=.09$& $\frac{9}{100}=.09$& $\frac{9}{100}=.09$&$ \frac{9}{100}=.09$& $\frac{9}{100}=.09$& $\frac{9}{100}=.09$\\
\hline
\end{tabular}
\end{flushleft}
\vspace{.25in}
\begin{flushleft}
\begin{tabular}{|r|c|c|c|c|c|}
\hline
$X$& $\del_{4}(6,11;X) $&$\del_{4}(7,11;X)$ &$ \del_{4}(8,11;X)$&$ \del_{4}(9,11;X)$&$\del_{4}(10,11;X)$\\
\hline
$10^4$       &$.0846$&$.0960$&$.1009$&$.1066$&$.0797$\\
$10^5$     &$.0902$&$.0925$&$.0955$&$.0948$&$.0914$\\
$10^6$  &$.0893$&$.0913$&$.0976$&$.0920$&$.0914$\\
$10^7$&$.0901$&$.0906$&$.0986$&$.0898$&$.0907$\\
$2\cdot10^7$&$.0898$&$.0902$&$.0991$&$.0897$&$.0906$\\
$5\cdot10^7$&$.0899$&$.0901$&$.0991$&$.0899$&$.0908$\\
\vdots& \vdots&\vdots &\vdots &\vdots &\vdots \\
$\infty$    & $\frac{9}{100}=.09$& $\frac{9}{100}=.09$& $\frac{119}{1200}\approx.0992$& $\frac{9}{100}=.09$&$\frac{109}{1200}\approx.0908$\\
\hline
\end{tabular}
\end{flushleft}
\end{table}

Notice that most congruence classes modulo $11$ asymptotically contain $9$ percent of all $A_d(p)$, but the classes of $8$ and $10$ have a slight excess, with $9.92\ldots$ percent and $9.08\ldots$ percent, respectively. As we shall see, this is a consequence of the fact that the coefficients of the logarithmic derivative of $\hcal_4(j(z))$ can be written modulo $11 $ as a function of the coefficients of a Hecke eigenform, and by a deep theorem of Deligne, the latter induces a Galois representation whose traces are the corresponding coefficients, which turn out to be dictated by group theory. Hence the given distribution of $A_d(p)$ modulo $11$ is actually a statement about the distribution of the conjugacy classes in the image of the representation, which is given by the Chebotarev Density Theorem.

Let $S_{\ell+1}$ be the space of cusp forms of weight $\ell+1$ on $\SL_2(\Z)$. Then we have the following congruence for $A_d(n)$ modulo $\ell$.
\begin{theorem}\label{camain}
\label{formula}
Suppose $\ell$ is prime and $-d<0$ and $D>0$ are fundamental discriminants such that $dD<\ell$ is fundamental, $\ell$ is inert in $\Q(\sqrt{-Dd}),$ $\left(\frac{\ell}{Dd}\right)=1,$ and $\ell\nmid n$. Also let $\nu(m)$ be the Dirichlet inverse of $\sum\limits_{k=1}^{D-1}\left(\frac Dk\right)\zeta_D^{km}$. Then there exist constants $c_0,c_1\cdots c_r$ depending only on $d$ and $\ell$ such that
\begin{align}
A_{D,d}(n)\equiv \frac1n\sum_{m|n}\nu(n/m)(-24c_0\sigma_1(m)+c_1a_1(m)+\dotsb+c_ra_r(m))\spmod{\ell}
\end{align}
where $r$ is the dimension of $S_{\ell+1}$ and $f_i(z)=\sum_{n=1}^{\infty}a_i(n)q^n$ are the normalized Hecke eigenforms in $S_{\ell+1}$.
\end{theorem}
\begin{remark}
When $D=1$, we have $\nu(n/m)=\mu(n/m)$, so this formula simplifies. A more general version of this theorem can be obtained when the conditions $\ell>d$ and $\left(\frac{\ell}{d}\right)$ are not necessarily met. The more general conditions are made clear in Section 3. Although at first glance, this strategy appears to give no information in the case $\ell|n$, there may be a similar result in such cases if we instead look modulo powers of $\ell$ exceeding the maximum power of $\ell$ dividing $n$.
\end{remark}
When $\ell\in\{5,7,13\}$ (the primes for which $\dim_{\C}S_{\ell+1}=0$) and $D=1$ this implies the trivial congruences 
\[
A_d(n)\equiv -24h(d)\spmod \ell.
\]
For $\ell\in\{11,17,19\}$, we have $\dim_{\C}S_{\ell+1}=1$. If we let $D=1$, and $n=p\neq \ell$, the formula simplifies to
\[
A_d(p) \equiv -24h(d)+\frac{c_2}{p}(a_1(p)-1)\spmod{\ell}.
\]

These formulas relate Borcherds product exponents to explicit modular elliptic curves. For example, if $d=4$ and $\ell=11$, $S_{12}$ is spanned by the weight $12$ cusp form with Fourier expansion
$$\Delta(z)=q-24q^2+252q^3-1472q^4+4830q^5-6048q^6+\cdots=\sum_{n=1}^{\infty}\tau(n)q^n,$$  so in this case $a_1(m)=\tau(m).$ 
We have that $$\Delta(z)\equiv \eta(z)^2\eta(11z)^2=\sum_{n=1}^{\infty}a_{X_0(11)}(n)q^n \pmod{11},$$ which is the modular form corresponding to the elliptic curve $X_0(11)$, which in turn implies that
$$a_{X_0(11)}(p)=p+1-\#\left(X_0(11)/\Fp\right).$$  
Thus from Theorem~\ref{formula}, the values of $A_d(p)$ modulo $11$ correspond to the number of points on $X_0(11)$ over $\Fp$, which has defining equation
$$y^2+y=x^3-x^2-10x-20.$$ 

Therefore, Table~\ref{asymptotic} reflects the distribution of the pairs $(p,a_{X_0(11)}(p))$ modulo $11$. Serre proved (see Theorem 11 of \cite{Serre}) that every modular form of weight $\ell+1$ modulo $\ell$ on $\SL_2(\Z)$ corresponds to a weight $2$ cusp form on $\Gamma_0(\ell)$. Thus we can see that the same phenomenon occurs for all pairs $(d,\ell)$ given here.
\begin{table}[htbp]\caption{}
\begin{tabular}{|c|l|}
\hline
$\ell$& $d$ \\
\hline
11& 3, 4, 11, 12, 15, 20, 67, 115, 148, 163, 267\\
17& 3, 7, 11, 12, 24, 28, 88, 91, 163, 267, 403\\
19& 4, 7, 11, 19, 20, 28, 35, 43, 163, 187, 235, 427\\
\hline
\end{tabular}
\label{smalld}
\end{table}

In these cases, the curves $X_0(17)$ and $X_0(19)$ are given by the defining equations $$X_0(17):y^2+xy+y=x^3-x^2-6x-4$$ and $$X_0(19):y^2+y=x^3+x^2-9x-15.$$

To prove these congruences, we will consider the logarithmic derivative of Zagier's twisted Hilbert class polynomials and use their relationship to supersingular polynomials to prove that they lie in $\Mt_{\ell+1},$ the space of weight $\ell+1$ modular forms modulo $\ell$. Dorman's work on differences of singular moduli in~\cite{Dorman} will be useful here. Using this, we find a formula for the exponents $A_{D,d}(n)$ in terms of the coefficients of Eisenstein series and cusp forms that form a basis for $S_{\ell+1}$. 

In Section~\ref{zagier}, we rigorously revisit Borcherds product expansions and their generalizations by Zagier. In Section~\ref{prelim}, we recall essential properties of Eisenstein series, Hilbert class polynomials, and supersingular polynomials.  In Section~\ref{proof}, we use these properties to prove Theorem~\ref{formula}. In Section~\ref{example}, we use the theory of Galois representations and the Chebotarev Density Theorem to determine the distribution of $p,a_1(p),\dotsc,a_k(p)$ and we give two examples by applying this procedure to two specific cases, $d=4,\ell=11$ and $d=20,\ell=31$.

\section*{Acknowledgements}

We would like to thank Ken Ono for his guidance and encouragement throughout this project.

\section{A theorem of Zagier}\label{zagier}

Define $M_{1/2}^{!}(\Gamma_{0}(4))$ to be the space of all weight $1/2$ modular forms on $\Gamma_0(4)$ which are meromorphic at the cusps and holomorphic everywhere else. It turns out that the exponents in the infinite product expansions of $\hcal_{d}(x)$ and $\hcal_{D,d}(x)$ are the coefficients of special modular forms in the ``plus space'' $M_{1/2}^{!}$, consisting of elements $\sum{a(n)q^{n}}$ of $M_{1/2}^{!}(\Gamma_{0}(4))$ such that $a(n)$ is nonzero only for $n\equiv 0$ or $1\pmod{4}$ and a finite number of $n<0$. While we will not be using the modularity of these forms directly, we will be using their coefficients as a way to study Hilbert class polynomials. There is a unique basis of $M_{1/2}^{!}$ consisting of the forms $f_d(z)$ whose coefficients are supported at $0,1$ modulo $4$. For every nonnegative integer $d$ congruent to $0$ or $3$ modulo $4$, $f_d(z)$ has a Fourier expansion of the form $q^{-d}+\sum\limits_{n=1}^{\infty}A(n,d)q^n$. The Fourier expansions of the first few $f_d$ are
\begin{align*}
f_0(z) &= 1+2q+2q^4+2q^9+2q^{16}+2q^{25}+2q^{36}+2q^{49}+2q^{64}+2q^{81}+\cdots,\\
f_3(z) &= q^{-3}-248q+26572q^4-85995q^5+1707264q^8-4096248q^9+\cdots,\\
f_4(z) &= q^{-4}+492q+143376q^4+565760q^5+18473000q^8+51180012q^9+\cdots,\\
f_7(z) &= q^{-7}-4119q+8288256q^4-52756480q^5+5734772736q^8+\cdots.
\end{align*}

Borcherds' infinite product expansion \cite{Borcherds} then takes the specific form
\[
\hcal_d(j(z))=q^{-h(d)}\prod_{n=1}^{\infty}(1-q^n)^{A(n^2,d)}.
\]

Just as Borcherds found a product expansion for $\hcal_d(x)$, Zagier \cite{Zagier} extended this to the twisted functions $\hcal_{D,d}(x)$ defined in (\ref{twists}).

\begin{theorem}[Zagier \cite{Zagier}]\label{Zag}
Let $D>1$ and $-d$ be a positive and negative discriminant, respectively, which are fundamental and relatively prime. 
Then we have a product expansion
\[
\hcal_{D,d}(j(z))=\prod_{n=1}^{\infty}P_D(q^n)^{A(n^2D,d)}
\]
where, if $\zeta_D=e^{2\pi i/D}$,
\[
P_D(t)=\prod_{k=1}^{D-1}(1-\zeta_D^kt)^{\left(\frac Dk\right)}.
\]
\end{theorem}
\begin{remark}
When $D=1$, we often have that $\hcal_{D,d}(x)=\hcal_d(x)$. Also, note that one can recover Borcherds' result from
Theorem $\ref{Zag}$ with only minor adjustments. Lastly, note that the value of $A_{D,d}(n)$ in the introduction is equivalent to $A(n^2D,d)$.
\end{remark}

\section{Modular forms modulo $\ell$}
\label{prelim}
To compute formulas for the exponents $A(n^2D,d),$ we will take the logarithmic derivative of the twisted function $\hcal_{D,d}(x)$ and prove that its reduction modulo $\ell$ is a modular form of weight $\ell +1$. In order to do this, we must recall several preliminary facts about modular forms and their reductions modulo $\ell$.

We first recall that the Eisenstein series $E_{2k}(z)$ has $q$-series expansion
\[
E_{2k}(z)=1-\frac{4k}{B_{2k}}\sum_{n=1}^{\infty}\sigma_{2k-1}(n)q^n
\]
where $B_m$ is the $m^{th}$ Bernoulli number and $\sigma_m(n)=\sum_{d|n}d^m$. For $k>1,$ $E_{2k}(z)$ is a modular form of weight $2k$. 
We will make use of the following fact about Eisenstein series which follows trivially from congruence properties of the Bernoulli numbers.

\begin{lemma}[See Lemma 1.22 in \cite{Ono}]\label{Emod}
Let $\ell$ be an odd prime. Then $E_{\ell-1}(z)\equiv 1 \pmod \ell$ and $E_{\ell+1}(z)\equiv E_2(z) \pmod \ell$.
\end{lemma}

We can use the supersingular polynomial modulo $\ell$ to obtain congruences for $\hcal_d(x)$, since for most discriminants $d$ the two polynomials share many roots. 
An elliptic curve $E$ over $\overline{\mathbb{F}}_{\ell}$ is called \emph{supersingular} if the group $E(\overline{\mathbb{F}}_{\ell})$ has no $p$-torsion. The supersingular polynomial $s_{\ell}(x)$ is given by
\[
s_{\ell}(x)=\prod_{\underset{E/\overline{\mathbb{F}}_{\ell}}{ E\text{ supersingular}}}(x-j(E))
\]
where $j(E)$ is the $j$-invariant of the curve $E$.

Using this, we can give conditions for when we can write the logarithmic derivative of $\hcal_{D,d}(j(z))$ as an element of $\Mt_{\ell+1}$.

\begin{theorem}\label{main}
Let $-d$ and $D$ be fundamental discriminants such that $-Dd$ is fundamental, and let $\ell$ be prime such that $\hcal_{Dd}(x)|s_{\ell}(x)$ in $\Fl[x]$. Then we have
$$-\frac{(\hcal_{D,d}(j(z)))'}{\hcal_{D,d}(j(z))}\in \Mt_{\ell+1}.$$
\end{theorem}
\begin{proof}
It is easy to see that the denominator of $-\frac{(\hcal_{D,d}(j(z)))'}{\hcal_{D,d}(j(z))}$ is exactly $\hcal_{Dd}(j(z))$. We want to construct a modular form $M(z)$ of weight $\ell-1$ that is congruent to $1$ modulo $\ell$ such that every pole introduced by the denominator of $\hcal_{Dd}(j(z))$ is canceled by a zero of $M(z)$. 

We write $\ell-1$ uniquely as $12m+4\del+6\eps$ where $m\in\Z_+$, $\delta\in\{0,1,2\}$, and $\eps\in\{0,1\}$. Write 
$$E_{\ell-1}(z)=\Del^m(z)E_4^{\del}(z)E_6^{\eps}(z)\Et_{\ell-1}(j(z))$$
where $\Et_{\ell-1}(j(z))$ is a polynomial in $j(z)$.
From Theorem~1 of \cite{Kaneko}, we have that
$$s_{\ell}(j(z))\equiv\pm j(z)^{\del}(j(z)-1728)^{\eps}\Et_{\ell-1}(j(z))\pmod \ell.$$
  Since we have assumed that $\hcal_{Dd}(x)|s_{\ell}(x)$ in $\Fl[x]$, we can 
construct $P(x)\in\Q[x]$ such that 
\[
\mathcal{H}_{Dd}(x)P(x)\equiv\pm s_{\ell}(x) \pmod \ell,
\]
 where the sign matches that of $s_{\ell}(j(z))$ above.


Define $M(z)$ to be the weight $\ell-1$ modular form given by
\[
M(z)=\Del^m(z)E_4^{\del}(z)E_6^{\eps}(z)\frac{\mathcal{H}_{Dd}(j(z))P(j(z))}{j(z)^{\del}(j(z)-1728)^{\eps}}.
\]
Then we can conclude that
$$M(z)\equiv\Del^m(z)E_4^{\del}(z)E_6^{\eps}(z)\Et_{\ell-1}(j(z))= E_{\ell-1}(z)\equiv 1\pmod \ell$$ by Lemma~\ref{Emod}.

Since $E_4(z)$ vanishes when $j(z)=0$ and $E_6(z)$ vanishes when $j(z)=1728$, we see that $M(z)=\Del^m(z)E_4^{\del}(z)E_6^{\eps}(z)\frac{\mathcal{H}_{Dd}(j(z))P(j(z))}{j(z)^{\del}(j(z)-1728)^{\eps}}$ vanishes at all the roots of $\hcal_{Dd}(j(z))$ to at least the same order. Also, $M(z)$ is holomorphic everywhere on the upper half-plane, including at infinity, since the factor $\mathcal{H}_{Dd}(j(z))P(j(z))$ has degree in $j(z)$ equal to the degree of $s_{\ell}$, which is $m$. Thus 
$$-\frac{(\hcal_{D,d}(j(z)))'}{\hcal_{D,d}(j(z))}M(z)$$
is a holomorphic modular form of weight $\ell+1$ that is congruent to $-\frac{(\hcal_{D,d}(j(z)))'}{\hcal_{D,d}(j(z))}\pmod \ell$, so we are done.
\end{proof}

The necessary condition of $\hcal_{D,d}(x)|s_{\ell}(x)$ in $\Fl[x]$ will occur for infinitely many pairs $(\ell,d)$. In fact, we can prove that the following conditions also imply the same result.
\begin{corollary}\label{equivcond}
Let $-d$ and $D$ be fundamental discriminants such that $-Dd$ is fundamental and let $\ell>Dd$ be a prime that is inert in $\Q(\sqrt{-Dd})$ and $\left(\frac{\ell}{Dd}\right)=1.$ Then we have
$$-\frac{(\hcal_{D,d}(j(z)))'}{\hcal_{D,d}(j(z))}\in \Mt_{\ell+1}.$$
\end{corollary}
\begin{proof}
From Theorem 7.25 in \cite{Ono}, we have that $$\hcal_{Dd}(x)|s_{\ell}(x)^{h(-Dd)}$$ in $\Fl[x]$, so every root of $\hcal_{Dd}(x)$ is also a root of $s_{\ell}(x)$. Dorman proves in Corollary~5.6 in \cite{Dorman} (see also \cite{Gross}) that the only primes $\ell$ for which $\hcal_{Dd}(x)$ can have a repeated root modulo $\ell$ are those where $\left(\frac{\ell}{Dd}\right)\neq1,$ which is not the case by assumption. Thus every root of $\hcal_{Dd}(x)$ has multiplicity $1$, implying 
$$\hcal_{Dd}(x)|s_{\ell}(x),$$ and by Theorem~\ref{main}, the result follows.
\end{proof}
We wish to massage this result into a form similar to what we had mentioned in the introduction; namely, that we can write the logarithmic derivative of $\hcal_d(j(z))$ as a multiple of $E_2(z)$ plus several cusp forms that form a basis of $S_{\ell+1}$. There is a correspondence between elements of $S_{\ell+1}$ and elements of $S_2(\Gamma_0(\ell))$ which we illustrate here.
\begin{lemma}\label{serrelemma}
An element $M(z)$ of $\Mt_{\ell+1}$ can be written uniquely modulo $\ell$ as the sum of a constant multiple of $E_2(z)$ and an element of $S_2(\Gamma_0(\ell))$.
\end{lemma}

\begin{proof}
Let $c$ be the constant coefficient of $M(z)$. Then $c$ is the unique element of $\Fl$ so that $M(z)-cE_{\ell+1}(z)\in \St_{\ell+1}$. We have from Lemma~\ref{Emod} that $E_{\ell+1}(z)\equiv E_2(z)\pmod{\ell}$, and by Theorem 11 of~\cite{Serre} that $M(z)-cE_{\ell+1}(z)$ corresponds modulo ${\ell}$ to an element of $S_2(\Gamma_0(\ell))$.
\end{proof}
We are now prepared to prove the types of congruences that we see in Theorem~\ref{formula}.

\section{Proofs of Congruences}
\label{proof}

We have shown that the logarithmic derivative of $\hcal_{D,d}(j(z))$ modulo $\ell$, which we can compute using Theorem~\ref{Zag}, is an element of $\Mt_{\ell+1}$, the reduction modulo $\ell$ of weight $\ell+1$ holomorphic modular forms with integer coefficients. In order to prove Theorem~\ref{formula}, we just need to solve for the exponents of $\hcal_{D,d}(j(z))$.

\begin{lemma}\label{twistlogder}Let $\ell$ be prime, $D>0$ and $-d<0$ be fundamental discriminants, and $g(n)$ be defined by $$-\frac{(\hcal_{D,d}(j(z)))'}{\hcal_{D,d}(j(z))}\equiv\sum_{n=0}^{\infty}g(n)q^n\pmod{\ell}.$$ Then for all $n$ not divisible by $\ell$,
$$A(n^2D,d) \equiv\frac{1}{n}\sum\limits_{m|n}\nu(m)g\left(\frac nm\right)\pmod \ell$$
where $\nu(m)$ is the Dirichlet inverse of $\sum\limits_{k=1}^{D-1}\left(\frac Dk\right)\zeta_D^{km}$.
\end{lemma}
\begin{proof}
From Theorem~\ref{Zag} we can compute
\begin{align*}
-\frac{(\hcal_{D,d}(j(z)))'}{\hcal_{D,d}(j(z))} &= -q\sum_{n=1}^{\infty}A(n^2D,d)\frac{d}{dq}\log P_D(q^n)\\
&= -q\sum_{n=1}^{\infty}A(n^2D,d)\sum_{k=1}^{D-1}\left(\frac Dk\right)\frac{-\zeta_D^k nq^{n-1}}{1-\zeta_D^kq^n}\\
&=\sum_{n=1}^{\infty}nA(n^2D,d)\sum_{k=1}^{D-1}\left(\frac Dk\right)(\zeta_D^kq^n+\zeta_D^{2k}q^{2n}+\zeta_D^{3k}q^{3n}+\dotsb)\\
&=\sum_{n=1}^{\infty}\sum_{m|n}mA(m^2D,d)\sum_{k=1}^{D-1}\left(\frac Dk\right)\zeta_D^{kn/m}q^n.
\end{align*}
Write $f_1(m)=mA(m^2D,d)$ and $f_2(m)=\sum_{k=1}^{D-1}\left(\frac Dk\right)\zeta_D^{km}$. Then 
\[
f_1*f_2=\sum_{m|n}mA(m^2D,d)\sum_{k=1}^{D-1}\left(\frac Dk\right)\zeta_D^{kn/m},
\]
where $*$ is Dirichlet convolution. From the theory of Gauss sums (see~\cite{Ireland}) we know that $f_2(1)\neq 0$, and so $f_2$ has a Dirichlet inverse $\nu$. Hence, since $\ell\nmid n,$ given a congruence of the form $-\frac{(\hcal_{D,d}(j(z)))'}{\hcal_{D,d}(j(z))}\equiv\sum\limits_{n=0}^{\infty}g(n)q^n\pmod{\ell}$ we can convolve $\nu$ with both sides to compute
\begin{align*}
\sum_{m|n}mA(m^2D,d)\sum_{k=1}^{D-1}\left(\frac Dk\right)\zeta_D^{kn/m} &\equiv g(n)\pmod \ell\\
nA(n^2D,d) &\equiv \sum_{m|n}\nu(m)g\left(\frac nm\right)\pmod \ell\\
A(n^2D,d) &\equiv\frac{1}{n}\sum\limits_{m|n}\nu(m)g\left(\frac nm\right)\pmod \ell.
\end{align*}
\end{proof}

This allows us to prove, among other things, the trivial congruences of $A(n^2,d)$ modulo $5$, $7$, and $13$ mentioned in the introduction.
\begin{example}\label{smallcong}
If $\ell\in\{5,7,13\}$ and $-d$ is a negative fundamental discriminant satisfying the conditions of Theorem~\ref{main} for $D=1$, then we have
$$A(n^2,d)\equiv-24h(d)\pmod \ell.$$
\end{example}
\begin{proof}
The assumed conditions let us use Theorem~\ref{main} to write $-\frac{(\hcal_{D,d}(j(z)))'}{\hcal_{D,d}(j(z))}$ as\\ 
$h(d)E_{\ell+1}$ plus some element of $S_{\ell+1}$. However, for the primes $5$, $7$, and $13$, the dimension of $S_{\ell+1}$ is $0$. Thus using Lemma~\ref{Emod}, we obtain 
\begin{eqnarray*}
-\frac{(\hcal_{D,d}(j(z)))'}{\hcal_{D,d}(j(z))}&\equiv& h(d)E_{\ell+1}\pmod \ell\\
&\equiv& h(d)E_2 \pmod \ell\\
&\equiv& h(d)-\sum_{n=1}^{\infty}24h(d)\sigma_1(n)q^n\pmod \ell.
\end{eqnarray*}
Thus by Lemma~\ref{twistlogder} with $D=1$ and the fact that $n=\sum_{m|n}\mu(m)\sigma_1(\frac{n}{m})$, we can conclude the desired identity.
\end{proof}

Lemma~\ref{twistlogder} allows us to complete the proof of Theorem~\ref{camain}.
\begin{proof}[Proof of Theorem~\ref{camain}]
Let $\nu(m)$ be the Dirichlet inverse of $\sum\limits_{k=1}^{D-1}\left(\frac Dk\right)\zeta_D^{km}$ as before. Since $\ell\nmid n$, Lemma~\ref{twistlogder} implies that $$A(n^2D,d)\equiv\frac{1}{n}\sum\limits_{m|n}\nu\left(\frac nm\right)g\left( m\right)\pmod \ell.$$
To find $g(m),$ we use the assumption that $\hcal_{Dd}(x)|s_{\ell}(x)$ (implied by the assumed conditions by the argument in Corollary~\ref{equivcond}) in $\Fl[x]$ to invoke Theorem~\ref{main}, implying that we can write $\sum\limits_{n=0}^{\infty}g(n)q^n$ as an element of $\Mt_{\ell+1}$. By Lemma~\ref{serrelemma}, this element can be represented as $c_0E_{\ell+1}$ plus an element of $\tilde{S}_{\ell+1}$. For some choice of size $r$ basis of cusp forms for $S_{\ell+1}$ with expansions $\sum\limits_{n=1}^{\infty}a_i(n)q^n$, we can use $E_{\ell+1}\equiv E_2\pmod {\ell}$ by Lemma~\ref{Emod} to imply that 
$$g(n)\equiv-24c_0\sigma_1(n)+c_1a_1(n)+c_2a_2(n)+\cdots+c_ra_r(n)\pmod {\ell}.$$ The desired result follows.
\end{proof}

\section{Applications of the Chebotarev Density Theorem}
\label{example}
From \cite{Ono}, we have the following corollary of a theorem of Deligne regarding Galois representations associated to certain modular forms.
\begin{theorem}
\label{formtorep}
Let $f(z)=\sum\limits_{n=1}^\infty a(n)q^n\in S_k^{\text{new}}(\Gamma_0(N),\chi)$ be a newform, and let $K_f$ be the number field obtained by adjoining the Fourier coefficients $a(n)$ and the values of $\chi$ to $\mathbb{Q}.$ Let $\ell$ be any prime, $K$ any finite extension of $\mathbb{Q}$ containing $K_f$, and $\mathfrak{p}_{\ell,K}$ a prime ideal of $\mathcal{O}_K$ dividing $\ell$. Then there is a continuous semisimple representation
\[
\rho_{f,\ell}:\gal(\overline{\mathbb{Q}}/\mathbb{Q})\to \GL_2(\mathbb{F}_{\ell,K})
\]
for which the following are true:
\begin{enumerate}
\item We have that $\rho_{f,\ell}$ is unramified at all primes $p\nmid N\ell.$
\item For every prime $p\nmid N\ell$ we have
\[
\tr(\rho_{f,\ell}(\frob_p))\equiv a(p)\pmod{\mathfrak{p}_{\ell,K}}.
\]
\item For every prime $p\nmid N\ell$ we have
\[
\det(\rho_{f,\ell}(\frob_p))\equiv\chi(p)p^{k-1}\pmod{\mathfrak{p}_{\ell,K}}.
\]
\item For any complex conjugation $c$, we have
\[
\det\rho_{f,\ell}(c)=-1.
\]
Here $\frob_p$ denotes any Frobenius element for the prime $p$.
\end{enumerate}
\end{theorem}

We will only be concerned with the case of modular forms with rational integer Fourier coefficients and $\chi$ trivial. For $\ell=11,$ $17,$ or $19$ and $-d<0$ satisfying the hypothesis of Corollary ~\ref{equivcond}, in the case that the Galois representation for the cusp form furnished by Lemma ~\ref{serrelemma} surjects onto $\GL_2(\mathbb{F}_\ell)$, the Chebotarev Density Theorem tells us that every pair $(p,a(p))$ in $\mathbb{F}_{\ell}^{\times}\times\mathbb{F}_{\ell}$ occurs. Further, it yields the precise densities of each ordered pair $(p,a(p))$ in the space of possible ordered pairs. Our goal in the first subsection is to compute the relative sizes of unions of conjugacy class in $\GL_2(\mathbb{F}_\ell)$ corresponding to each $(p,a(p))$ modulo ${\ell}.$ Applying the Chebotarev Density Theorem yields our desired result. We then apply the result to the example where $d=4$ and $\ell=11.$
\subsection{Characteristic polynomial frequencies in $\GL_2(\Fl)$}
It will be useful to have a formula for the proportion of elements in $\GL_2(\mathbb{F}_\ell)$ with a given trace and determinant.

\begin{lemma}
\label{dens}
Of the $(\ell^2-1)(\ell^2-\ell)$ elements of $\GL_2(\Fl)$, the proportion with characteristic polynomial $x^2-ax+b$ (that is, trace $a$ and determinant $b$), for $a\in\Fl$ and $b\in\Fl^{\times}$, are as follows:
\begin{align}
\begin{cases}
\frac{1}{(\ell-1)(\ell+1)} & \text{ if } 
\left(\frac{-b+a^2/4}{\ell}\right)=-1\\ 
\frac{1}{(\ell-1)^2} & \text{ if } 
\left(\frac{-b+a^2/4}{\ell}\right)=1\\
\frac{\ell}{(\ell-1)^2(\ell+1)} & \text{ if } 
\left(\frac{-b+a^2/4}{\ell}\right)=0.
\end{cases}
\end{align}
\end{lemma}

\begin{remark}
Note that for $a=0$, the first two cases each occur $(\ell-1)/2$ times; and for a fixed $a\neq0$, the first, second, and third case respectively occur $(\ell-1)/2$ times, $(\ell-3)/2$ times, and once. 
\end{remark}

\begin{proof}


Let $g\in \GL_2(\mathbb{F}_\ell)$ have trace $a$, so that $g$ is of the form $\begin{pmatrix}a+c & m\\n & -c\end{pmatrix}$. Then $-c(a+c)-mn\equiv b\pmod{\ell}$, or $mn\equiv-b-c(a+c)\equiv-b+a^2/4-(c+a/2)^2\pmod{\ell}$. If $-b+a^2/4$ is not a quadratic residue modulo $\ell$, then for each of the $\ell$ choices of $c$ we have $\ell-1$ choices of $m$, after which $n$ is fixed, giving a count of $\ell(\ell-1)$ ordered triples $(c,m,n)$. If $-b+a^2/4$ is a nonzero quadratic residue modulo $\ell$, then for each of the $\ell-2$ choices of $c$ so that $-b+a^2/4-(c+a/2)^2\not\equiv0\pmod{\ell}$ we still have $\ell-1$ choices of $m$ after which $n$ is fixed, but if $-b+a^2/4-(c+a/2)^2\equiv0\pmod{\ell}$ then we have $2\ell-1$ choices of the ordered pair $(m,n)$. This gives a count of $(\ell-2)(\ell-1)+2(2\ell-1)=\ell(\ell+1)$ triples $(c,m,n)$. Finally, if $-b+a^2/4\equiv0\pmod{\ell}$, by similar reasoning, we have $(\ell-1)^2+2\ell-1=\ell^2$ triples $(c,m,n)$.
\end{proof}

Using Theorem 1.34 in \cite{Ono}, it is straightforward to show that the primes $\ell$ for which $\dim_\C S_2(\Gamma_0(\ell))=1$ are exactly $11,$ $17,$ and $19.$ 
The spaces $S_2(\Gamma_0(11)),$ $S_2(\Gamma_0(17)),$ and $S_2(\Gamma_0(19))$ are generated by the modular form associated with the elliptic curves $X_0(11),$ $X_0(17)$, and $X_0(19)$, respectively.
Because the dimensions of these spaces is $1$, the associated modular forms are Hecke eigenforms, and hence their Dirichlet series each possess an Euler product by Theorem 6.19 of \cite{Apostol}. Thus, the result of Deligne applies. 
In Proposition~19 of \cite{Serre2}, Serre provides sufficient criteria for when a subgroup $G$ of $\GL_2(\mathbb{F}_{\ell})$ is all of $\GL_2(\mathbb{F}_{\ell}).$
This criteria requires the existence of merely three elements of $G$ with trace and determinant satisfying certain conditions, along with the hypothesis that $\det:G\mapsto\mathbb{F}_{\ell}^\times$ should be surjective. 
Of course, since $\det\rho_\ell(\frob_p)=p$ in this case, the surjectivity of $\det$ is trivially guaranteed. Using this result, we only need to check the image of $\frob_p$ for primes $p\leq7$ in order to show that  $\rho_{11}$, $\rho_{17}$, and $\rho_{19}$ are surjective. Thus, given $\rho_{\ell}$, $\ell\in\lbrace 11,17,19\rbrace$, and an appropriate $d$ from Table~\ref{smalld}, we can use Lemma~\ref{dens}, Theorem~\ref{formula}, and the Chebotarev Density Theorem to compute the densities in the limit of each congruence class modulo $\ell$ for $A(p^2,d)$ where $p$ is prime.

\begin{theorem}[The Chebotarev Density Theorem] Let $L/K$ be Galois and let $C\subset\gal(L/K)$ be a conjugacy class. Then 
\[
\{\mathfrak{p}:\mathfrak{p}\text{ a prime of }K,\mathfrak{p}\nmid\Delta_{L/K},\frob_\mathfrak{p}\in C\}
\]
has arithmetic density $\frac{|C|}{|G|}.$
\end{theorem}
We apply the Chebotarev Density Theorem with $K=\Q$ and $L$ the fixed field of the kernel of $\rho_\ell$, so that $\gal(L/K)\cong \GL_2(\Fl)$. Because the cycle type of $\frob_p$ corresponds exactly to the splitting type of $\mathfrak{p}$ when $\mathfrak{p}$ lies above $p,$ the Chebotarev Density Theorem gives us that the $p\neq \ell$ for which $a(p)\equiv a\pmod{\ell}$ and $p\equiv b\pmod{\ell}$ has density the relative size of the union of those conjugacy classes in $\GL_2(\mathbb{F}_\ell)$ with characteristic polynomial $X^2-aX+b.$

\subsection{Example:} Let $d=4,\ell=11$.
To compute the congruence relation for $A(n^2,4),$ we consider the logarithmic derivative of $\hcal_{4}(j(z))$,
\begin{eqnarray*}
-\frac{(\hcal_{4}(j(z)))'}{\hcal_{4}(j(z))}&=&-q\frac{d}{dq}\log\left(q^{-\frac12}\prod_{n=1}^{\infty}(1-q^n)^{A(n^2,4)}\right)\\
&=&\frac12+\sum_{n=1}^{\infty}\sum_{m|n}mA(m^2,4)q^n.
\end{eqnarray*}
By Theorem~\ref{main} and Lemma~\ref{Emod}, we know we can write this as a multiple of $E_{12}(z)\equiv E_2(z)$ and a cusp form modulo $11$ of weight $12$. Since $S_{12}$ is spanned by $\Delta(z),$ we can compare the coefficients of the constant, $q$, and $q^2$ terms on each side to obtain
$$-\frac{(\hcal_{4}(j(z)))'}{\hcal_{4}(j(z))}\equiv 6E_2(z)+9\Delta(z)\pmod {11}.$$
Thus by M\"obius inversion, we have a formula of the form of Theorem~\ref{formula}:
$$A(n^2,4)\equiv 10+\frac{9}{n}\sum_{m|n}\mu\left(\frac nm\right)\tau(m) \pmod{11}.$$
 
In the special case of $n=p$, this simplifies to
\[
A(p^2,4)\equiv10+9p^9\left(\tau(p)-1\right)\pmod{11}.
\]

Because the associated Galois representation $\rho_{11}$ is surjective, the densities in the limit of each congruence class modulo $11$ for $A(p^2,4)$ can be calculated from Lemma ~\ref{dens}. They appear in the last row of Table ~\ref{asymptotic}. This implies that for any $c\in\mathbb{F}_{11}$ there exist infinitely many primes $p$ for which $A(p^2,4)\equiv c\pmod{11}.$

\subsection{Example:} Let $d=20,\ell=31$.
Here the dimension of $S_{32}$ is $2$, so this example is more complicated. To compute the congruence relation for $A(n^2,20),$ we consider the logarithmic derivative of $\hcal_{20}(j(z))$,
\begin{eqnarray*}
-\frac{(\hcal_{20}(j(z)))'}{\hcal_{20}(j(z))}&=&-q\frac{d}{dq}\log\left(q^{-2}\prod_{n=1}^{\infty}(1-q^n)^{A(n^2,20)}\right)\\
&=&2+\sum_{n=1}^{\infty}\sum_{m|n}mA(m^2,20)q^n.
\end{eqnarray*}
Since it is straightforward to verify that $\hcal_{20}(x)|s_{31}(x)$ in $\mathbb F_{31}$, by Theorem~\ref{main} and Lemma~\ref{Emod} we know we can write this as a multiple of $E_{32}(z)\equiv E_2(z)$ and a cusp form modulo $31$ of weight $32$.

We choose the basis $\Delta^2(z)E_4^2(z)$, $\Delta(z)E_4^2(z)E_6^2(z)$ for $S_{32}$,
and so by comparing the coefficients of the constant, $q$, and $q^2$ terms on each side, we obtain
\begin{align}
\label{initial}
-\frac{(\hcal_{20}(j(z)))'}{\hcal_{20}(j(z))}\equiv 2E_2(z)+14\Delta^2(z)E_4^2(z)+23\Delta(z) E_4^2(z) E_6^2(z)\pmod {31}.
\end{align}
We now rewrite this in terms of simultaneous normalized Hecke eigenforms:
\begin{align*}
F_1(z)=&\Delta(z) E_4^2(z) E_6^2(z)+\frac{1711+\sqrt{18295489}}{184415616}\Delta^2(z)E_4^2(z)\text{ and}\\
F_2(z)=&\Delta(z) E_4^2(z) E_6^2(z)+\frac{1711-\sqrt{18295489}}{184415616}\Delta^2(z)E_4^2(z),
\end{align*}
 which are Galois conjugates. Since they are already normalized, Theorem~\ref{formtorep} applies. Now we conveniently chose the modulus $31$ so that these eigenforms are defined over $\mathbb{F}_{31}$. They are
\begin{align*}
F_1(z)&=\Delta(z) E_4^2(z) E_6^2(z)+ 22\Delta^2(z)E_4^2(z)\text{ and}\\
F_2(z)&=\Delta(z) E_4^2(z) E_6^2(z)+19\Delta^2(z)E_4^2(z),
\end{align*}
in terms of which Equation~\ref{initial} may be rewritten
\[
-\frac{(\hcal_{20}(j(z)))'}{\hcal_{20}(j(z))}\equiv 2E_2(z)+14F_1(z)+9F_2(z)\pmod {31}.
\]
Letting $F_1(z)=\sum a_1(n)q^n$ and $F_2(z)=\sum a_2(n)q^n$, we apply M\"obius inversion to obtain
\[
A(n^2,20)\equiv 14+\frac{1}{n}\sum_{m|n}\mu\left(\frac{n}{m}\right)(14a_1(m)+9a_2(m)) \pmod{31}
\]
as in Theorem~\ref{formula}.

If $\rho_1$ and $\rho_2$ are the Galois representations associated with $F_1$ and $F_2$ respectively, it is easily checked that
\[
\left\{\begin{pmatrix}
\rho_1(\frob_p)&0\\
0&\rho_2(\frob_p)\\
\end{pmatrix}:p\neq 31\right\}=\left\{\begin{pmatrix}
M&0\\
0&N\\
\end{pmatrix}:M,N\in GL_2(\mathbb{F}_p),\det N=\det M\right\}
\]
using the results of Serre \cite{Serre2} and Ribet \cite{Ribet}.

By a computation similar to the one in the $d=4,$ $\ell=11$ case, we get
\[
\delta_{20}(t,31;\infty)=\begin{cases}
\frac{991}{29760}&\mbox{ if }t=0 \\
\frac{1199}{37200}&\mbox{ if }t=1,2,9,14,21,29 \\
\frac{29}{900}&\mbox{ if }t=3,4,5,11,16,19,20,23,28 \\
\frac{719}{22320}&\mbox{ if }t=6,7,10,18,25,30 \\
\frac{14399}{446400}&\mbox{ if }t=8 \\
\frac{799}{24800}&\mbox{ if }t=12,13,15,17 \\
\frac{7193}{223200}&\mbox{ if }t=22,24,26,27.
\end{cases}.
\]

\end{document}